\documentclass{amsart}
\usepackage{amstext,amssymb,amsthm,amsopn,newlfont,graphpap,graphics,graphicx,mathrsfs}
\allowdisplaybreaks
\usepackage[parfill]{parskip}
\usepackage[noadjust]{cite}
\usepackage{epigraph}
\usepackage[colorlinks=true,
            linkcolor=red,
            urlcolor=blue,
            citecolor=magenta]{hyperref}
\usepackage{color}
\usepackage{mathrsfs}
\allowdisplaybreaks
\theoremstyle{plain}
\newtheorem{thm}{Theorem}[section]
\newtheorem*{thm*}{Theorem}
\newtheorem{prop}{Proposition}[section]
\newtheorem*{prop*}{Proposition}
\newtheorem{cor}{Corollary}[section]
\newtheorem*{cor*}{Corollary}

\newtheorem*{lem*}{Lemma}
\theoremstyle{definition}
\newtheorem{defn}{Definition}[section]
\newtheorem*{defn*}{Definition}

\newtheorem*{exmp*}{Example}

\newtheorem*{exmps*}{Examples}

\newtheorem{rem}{Remark}[section]
\newtheorem*{rem*}{Remark}
\newtheorem{rems}{Remarks}[section]
\newtheorem*{rems*}{Remarks}

\newtheorem*{note*}{Note}
\newcommand{\N}{{\mathbb N}}
\newcommand{\Z}{{\mathbb Z}}
\newcommand{\R}{{\mathbb R}}
\newcommand{\C}{{\mathbb C}}

\DeclareMathOperator{\Rep}{Re\,}
\DeclareMathOperator{\Imp}{Im\,}
\DeclareMathOperator{\dist}{dist}
\DeclareMathOperator{\spa}{span}
\begin{document}
\title[On the differentiability of weak solutions]
{On the differentiability of weak solutions\\
of an abstract evolution equation\\
with a scalar type spectral operator\\
on the real axis}
\author[Marat V. Markin]{Marat V. Markin}
\address{
Department of Mathematics\newline
California State University, Fresno\newline
5245 N. Backer Avenue, M/S PB 108\newline
Fresno, CA 93740-8001, USA
}
\email{mmarkin@csufresno.edu}
\dedicatory{}
\begin{abstract}
Given the abstract evolution equation
\begin{equation*}
y'(t)=Ay(t),\ t\in \R,
\end{equation*}
with \textit{scalar type spectral operator} $A$ in a complex Banach space, found are conditions \textit{necessary and sufficient} for all \textit{weak solutions} of the equation, which a priori need not be strongly differentiable, to be strongly infinite differentiable on $\R$. The important case of the equation with a \textit{normal operator} $A$ in a complex Hilbert space is obtained immediately as a particular case. Also, proved is the following inherent smoothness improvement effect explaining why the case of the strong finite differentiability of the weak solutions is superfluous: if every weak solution of the equation is strongly differentiable at $0$, then all of them are strongly infinite differentiable on $\R$.
\end{abstract}
\maketitle
\epigraph{\textit{Curiosity is the lust of the mind.}}{Thomas Hobbes}
\section[Introduction]{Introduction}

We find conditions on a \textit{scalar type spectral} operator $A$ in a complex Banach space necessary and sufficient for all \textit{weak solutions} of the evolution equation
\begin{equation}\label{1}
y'(t)=Ay(t),\ t\in \R,
\end{equation}
which a priori need not be strongly differentiable,
to be strongly infinite differentiable on $\R$.
The important case of the equation with a \textit{normal operator} $A$ in a complex Hilbert space is obtained immediately as a particular case.
We also prove the following inherent smoothness improvement effect explaining why the case of the strong finite differentiability of the weak solutions is superfluous: if every weak solution of the equation is strongly differentiable at $0$, then all of them are strongly infinite differentiable on $\R$.

The found results develop those of paper \cite{Markin2011}, where similar consideration is given to the strong differentiability of the weak solutions of the equation
\begin{equation}\label{+}
y'(t)=Ay(t),\ t\ge 0,
\end{equation}
on $[0,\infty)$ and $(0,\infty)$.

\begin{defn}[Weak Solution]\label{ws}\ \\
Let $A$ be a densely defined closed linear operator in a Banach space $X$ and $I$ be an interval of the real axis $\R$. A strongly continuous vector function $y:I\rightarrow X$ is called a {\it weak solution} of
the evolution equation 
\begin{equation}\label{2}
y'(t)=Ay(t),\ t\in I,
\end{equation}
if, for any $g^* \in D(A^*)$,
\begin{equation*}
\dfrac{d}{dt}\langle y(t),g^*\rangle = \langle y(t),A^*g^* \rangle,\ t\in I,
\end{equation*}
where $D(\cdot)$ is the \textit{domain} of an operator, $A^*$ is the operator {\it adjoint} to $A$, and $\langle\cdot,\cdot\rangle$ is the {\it pairing} between
the space $X$ and its dual $X^*$ (cf. \cite{Ball}).
\end{defn}

\begin{rems}\label{remsws}\
\begin{itemize}
\item Due to the \textit{closedness} of $A$, a weak solution of equation \eqref{2} can be equivalently defined to be a strongly continuous vector function $y:I\mapsto X$ such that, for all $t\in I$,
\begin{equation*}
\int_{t_0}^ty(s)\,ds\in D(A)\ \text{and} \ y(t)=y(t_0)+A\int_{t_0}^ty(s)\,ds,
\end{equation*}
where $t_0$ is an arbitrary fixed point of the interval $I$, and is also called a \textit{mild solution} (cf. {\cite[Ch. II, Definition 6.3]{Engel-Nagel}}, see also {\cite[Preliminaries]{Markin2018(2)}}).
\item Such a notion of \textit{weak solution}, which need not be differentiable in the strong sense, generalizes that of \textit{classical} one, strongly differentiable on $I$ and satisfying the equation in the traditional plug-in sense, the classical solutions being precisely the weak ones strongly differentiable on $I$.
\item As is easily seen $y:\R\to X$ is a weak solution of equation \eqref{1} \textit{iff} 
\[
y_+(t):=y(t),\ t\ge 0,
\]
is a weak solution of equation \eqref{+} and 
\[
y_-(t):=y(-t),\ t\ge 0,
\]
is a weak solution  of the equation
\begin{equation}\label{-}
y'(t)=-Ay(t),\ t\ge 0.
\end{equation}
\item When a closed densely defined linear operator $A$
in a complex Banach space $X$ generates a strongly continuous group $\left\{T(t) \right\}_{t\in \R}$ of  bounded linear operators (see, e.g., \cite{Hille-Phillips,Engel-Nagel}), i.e., the associated \textit{abstract Cauchy problem} (\textit{ACP})
\begin{equation}\label{ACP}
\begin{cases}
y'(t)=Ay(t),\ t\in \R,\\
y(0)=f
\end{cases}
\end{equation}
is \textit{well-posed} (cf. {\cite[Ch. II, Definition 6.8]{Engel-Nagel}}), the weak solutions of equation \eqref{1} are the orbits
\begin{equation}\label{group}
y(t)=T(t)f,\ t\in \R,
\end{equation}
with $f\in X$ (cf. {\cite[Ch. II, Proposition 6.4]{Engel-Nagel}}, see also {\cite[Theorem]{Ball}}), whereas the classical ones are those with $f\in D(A)$
(see, e.g., {\cite[Ch. II, Proposition 6.3]{Engel-Nagel}}). 
\item In our discourse, the associated \textit{ACP} may be \textit{ill-posed}, i.e., the scalar type spectral operator $A$ need not generate a strongly continuous group of  bounded linear operators (cf. \cite{Markin2002(2)}). 
\end{itemize} 
\end{rems} 

\section[Preliminaries]{Preliminaries}

Here, for the reader's convenience, we outline certain essential preliminaries.

Henceforth, unless specified otherwise, $A$ is supposed to be a {\it scalar type spectral operator} in a complex Banach space $(X,\|\cdot\|)$ with strongly $\sigma$-additive \textit{spectral measure} (the \textit{resolution of the identity}) $E_A(\cdot)$ assigning to each Borel set $\delta$ of the complex plane $\C$ a projection operator $E_A(\delta)$ on $X$ and having the operator's \textit{spectrum} $\sigma(A)$ as its {\it support} \cite{Survey58,Dun-SchIII}.

Observe that, in a complex finite-dimensional space, 
the scalar type spectral operators are all linear operators on the space, for which there is an \textit{eigenbasis} (see, e.g., \cite{Survey58,Dun-SchIII}) and, in a complex Hilbert space, the scalar type spectral operators are precisely all those that are similar to the {\it normal} ones \cite{Wermer}.

Associated with a scalar type spectral operator in a complex Banach space is the {\it Borel operational calculus} analogous to that for a \textit{normal operator} in a complex Hilbert space \cite{Survey58,Dun-SchII,Dun-SchIII,Plesner}, which assigns to any Borel measurable function $F:\sigma(A)\to \C$ a scalar type spectral operator
\begin{equation*}
F(A):=\int\limits_{\sigma(A)} F(\lambda)\,dE_A(\lambda)
\end{equation*}
(see \cite{Survey58,Dun-SchIII}).

In particular,
\begin{equation}\label{power}
A^n=\int\limits_{\sigma(A)} \lambda^n\,dE_A(\lambda),\ n\in\Z_+,
\end{equation}
($\Z_+:=\left\{0,1,2,\dots\right\}$ is the set of \textit{nonnegative integers}, $A^0:=I$, $I$ is the \textit{identity operator} on $X$) and
\begin{equation}\label{exp}
e^{zA}:=\int\limits_{\sigma(A)} e^{z\lambda}\,dE_A(\lambda),\ z\in\C.
\end{equation}

The properties of the {\it spectral measure} and {\it operational calculus}, exhaustively delineated in \cite{Survey58,Dun-SchIII}, underlie the entire subsequent discourse. Here, we underline a few facts of particular importance.

Due to its {\it strong countable additivity}, the spectral measure $E_A(\cdot)$ is {\it bounded} \cite{Dun-SchI,Dun-SchIII}, i.e., there is such an $M>0$ that, for any Borel set $\delta\subseteq \C$,
\begin{equation}\label{bounded}
\|E_A(\delta)\|\le M.
\end{equation}

Observe that the notation $\|\cdot\|$ is used here to designate the norm in the space $L(X)$ of all bounded linear operators on $X$. We adhere to this rather conventional economy of symbols in what follows also adopting the same notation for the norm in the dual space $X^*$. 

For any $f\in X$ and $g^*\in X^*$, the \textit{total variation measure} $v(f,g^*,\cdot)$ of the complex-valued Borel measure $\langle E_A(\cdot)f,g^* \rangle$ is a {\it finite} positive Borel measure with
\begin{equation}\label{tv}
v(f,g^*,\C)=v(f,g^*,\sigma(A))\le 4M\|f\|\|g^*\|
\end{equation}
(see, e.g., \cite{Markin2004(1),Markin2004(2)}).

Also (Ibid.), for a Borel measurable function $F:\C\to \C$, $f\in D(F(A))$, $g^*\in X^*$, and a Borel set $\delta\subseteq \C$,
\begin{equation}\label{cond(ii)}
\int\limits_\delta|F(\lambda)|\,dv(f,g^*,\lambda)
\le 4M\|E_A(\delta)F(A)f\|\|g^*\|.
\end{equation}
In particular, for $\delta=\sigma(A)$,
$E_A(\sigma(A))=I$ and
\begin{equation}\label{cond(i)}
\int\limits_{\sigma(A)}|F(\lambda)|\,d v(f,g^*,\lambda)\le 4M\|F(A)f\|\|g^*\|.
\end{equation}

Observe that the constant $M>0$ in \eqref{tv}--\eqref{cond(i)} is from 
\eqref{bounded}.

Further, for a Borel measurable function $F:\C\to [0,\infty)$, a Borel set $\delta\subseteq \C$, a sequence $\left\{\Delta_n\right\}_{n=1}^\infty$ 
of pairwise disjoint Borel sets in $\C$, and 
$f\in X$, $g^*\in X^*$,
\begin{equation}\label{decompose}
\int\limits_{\delta}F(\lambda)\,dv(E_A(\cup_{n=1}^\infty \Delta_n)f,g^*,\lambda)
=\sum_{n=1}^\infty \int\limits_{\delta\cap\Delta_n}F(\lambda)\,dv(E_A(\Delta_n)f,g^*,\lambda).
\end{equation}

Indeed, since, for any Borel sets $\delta,\sigma\subseteq \C$,
\begin{equation*}
E_A(\delta)E_A(\sigma)=E_A(\delta\cap\sigma)
\end{equation*}
\cite{Survey58,Dun-SchIII}, 
for the total variation measure,
\begin{equation*}
v(E_A(\delta)f,g^*,\sigma)=v(f,g^*,\delta\cap\sigma).
\end{equation*}

Whence, due to the {\it nonnegativity} of $F(\cdot)$ (see, e.g., \cite{Halmos}),
\begin{multline*}
\int\limits_\delta F(\lambda)\,dv(E_A(\cup_{n=1}^\infty \Delta_n)f,g^*,\lambda)
=\int\limits_{\delta\cap\cup_{n=1}^\infty \Delta_n}F(\lambda)\,dv(f,g^*,\lambda)
\\
\ \
=\sum_{n=1}^\infty \int\limits_{\delta\cap\Delta_n}F(\lambda)\,dv(f,g^*,\lambda)
=\sum_{n=1}^\infty \int\limits_{\delta\cap\Delta_n}F(\lambda)\,dv(E_A(\Delta_n)f,g^*,\lambda).
\hfill
\end{multline*}

The following statement, allowing to characterize the domains of Borel measurable functions of a scalar type spectral operator in terms of positive Borel measures, is fundamental for our discourse.

\begin{prop}[{\cite[Proposition $3.1$]{Markin2002(1)}}]\label{prop}\ \\
Let $A$ be a scalar type spectral operator in a complex Banach space $(X,\|\cdot\|)$ with spectral measure $E_A(\cdot)$ and $F:\sigma(A)\to \C$ be a Borel measurable function. Then $f\in D(F(A))$ iff
\begin{enumerate}
\item[(i)] for each $g^*\in X^*$, 
$\displaystyle \int\limits_{\sigma(A)} |F(\lambda)|\,d v(f,g^*,\lambda)<\infty$ and
\item[(ii)] $\displaystyle \sup_{\{g^*\in X^*\,|\,\|g^*\|=1\}}
\int\limits_{\{\lambda\in\sigma(A)\,|\,|F(\lambda)|>n\}}
|F(\lambda)|\,dv(f,g^*,\lambda)\to 0,\ n\to\infty$,
\end{enumerate}
where $v(f,g^*,\cdot)$ is the total variation measure of $\langle E_A(\cdot)f,g^* \rangle$.
\end{prop} 

The succeeding key theorem provides a description of the weak solutions of equation \eqref{+} with a scalar type spectral operator $A$ in a complex Banach space.

\begin{thm}[{\cite[Theorem $4.2$]{Markin2002(1)}} with $T=\infty$]\label{GWS+}\ \\
Let $A$ be a scalar type spectral operator in a complex Banach space $(X,\|\cdot\|)$. A vector function $y:[0,\infty) \to X$ is a weak solution 
of equation \eqref{+} iff there is an $\displaystyle f \in \bigcap_{t\ge 0}D(e^{tA})$ such that
\begin{equation*}
y(t)=e^{tA}f,\ t\ge 0,
\end{equation*}
the operator exponentials understood in the sense of the Borel operational calculus (see \eqref{exp}).
\end{thm}

\begin{rem}
Theorem \ref{GWS+} generalizes {\cite[Theorem $3.1$]{Markin1999}}, its counterpart for a normal operator $A$ in a complex Hilbert space.
\end{rem} 

We also need the following characterizations of a particular weak solution's of equation \eqref{+} with a scalar type spectral operator $A$ in a complex Banach space being strongly differentiable on a subinterval $I$ of $[0,\infty)$.

\begin{prop}[{\cite[Proposition $3.1$]{Markin2011}} with $T=\infty$]\label{Prop}\ \\
Let $n\in\N$ and $I$ be a subinterval of $[0,\infty)$.  A weak solution $y(\cdot)$ of equation \eqref{+} is $n$ times strongly differentiable on $I$ iff
\begin{equation*}
y(t) \in D(A^n),\ t\in I,
\end{equation*}
in which case, 
\begin{equation*}
y^{(k)}(t)=A^ky(t),\ k=1,\dots,n,t\in I.
\end{equation*}
\end{prop}

Subsequently, the frequent terms {\it ``spectral measure"} and {\it ``operational calculus"} are abbreviated to {\it s.m.} and {\it o.c.}, respectively.

\section{General Weak Solution}

\begin{thm}[General Weak Solution]\label{GWS}\ \\
Let $A$ be a scalar type spectral operator in a complex Banach space $(X,\|\cdot\|)$. A vector function $y:\R \to X$ is a weak solution 
of equation \eqref{1} iff there is an $\displaystyle f \in \bigcap_{t\in \R}D(e^{tA})$ such that
\begin{equation}\label{expf}
y(t)=e^{tA}f,\ t\in \R,
\end{equation}
the operator exponentials understood in the sense of the Borel operational calculus (see \eqref{exp}).
\end{thm}

\begin{proof}\quad 
As is noted in the Introduction, $y:\R\to X$ is a weak solution of \eqref{1} \textit{iff} 
\[
y_+(t):=y(t),\ t\ge 0,
\]
is a weak solution of equation \eqref{+} and 
\[
y_-(t):=y(-t),\ t\ge 0,
\]
is a weak solution of equation \eqref{-}.

Applying Theorem \ref{GWS+}, to $y_+(\cdot)$ and $y_-(\cdot)$, we infer that, this is  equivalent to the fact
\[
y(t)=e^{tA}f,\ t\in\R,\ \text{with some}\
f \in \bigcap_{t\in \R}D(e^{tA}).
\]
\end{proof}

\begin{rems}\
\begin{itemize}
\item More generally, Theorem \ref{GWS+} and its proof can be easily modified to describe in the same manner all weak solution of equation \eqref{2} for an arbitrary interval $I$ of the real axis
$\R$.
\item Theorem \ref{GWS} implies, in particular,
\begin{itemize}
\item that the subspace $\bigcap_{t\in\R}D(e^{tA})$ of all possible initial values of the weak solutions of equation \eqref{1} is the largest permissible for the exponential form given by \eqref{expf}, which highlights the naturalness of the notion of weak solution, and
\item that associated \textit{ACP} \eqref{ACP}, whenever solvable,  is solvable \textit{uniquely}.
\end{itemize}
\item Observe that the initial-value subspace $\bigcap_{t\in\R}D(e^{tA})$ of equation \eqref{1}, containing the dense in $X$ subspace $\bigcup_{\alpha>0}E_A(\Delta_\alpha)X$, where
\begin{equation*}
\Delta_\alpha:=\left\{\lambda\in\C\,\middle|\,|\lambda|\le \alpha \right\},\ \alpha>0,
\end{equation*}
which coincides with the class ${\mathscr E}^{\{0\}}(A)$ of \textit{entire} vectors of $A$ of \textit{exponential type} \cite{Markin2015}, is \textit{dense} in $X$ as well.
\item When a scalar type spectral operator $A$ in a complex Banach space generates a strongly continuous group $\left\{T(t) \right\}_{t\in \R}$ of bounded linear operators, 
\[
T(t)=e^{tA}\ \text{and}\ D(e^{tA})=X,\ t\in \R,
\]
\cite{Markin2002(2)}, and hence, Theorem \ref{GWS} is consistent with the well-known description of the weak solutions for this setup (see \eqref{group}).
\item Clearly, the initial-value subspace $\bigcap_{t\in\R}D(e^{tA})$ of equation \eqref{1} is narrower than the initial-value subspace $\bigcap_{t\ge 0}D(e^{tA})$ of equation \eqref{+} and the initial-value subspace $\bigcap_{t\ge 0}D(e^{t(-A)})=\bigcap_{t\le 0}D(e^{tA})$ of equation \eqref{-}, in fact it is the intersection of the latter two.
\end{itemize} 
\end{rems} 

\section{Differentiability 
of a Particular Weak Solution}

Here, we characterize a particular weak solution's of equation \eqref{1} with a scalar type spectral operator $A$ in a complex Banach space being strongly differentiable on a subinterval $I$ of $\R$.

\begin{prop}[Differentiability of a Particular Weak Solution]\label{particular}\ \\
Let $n\in\N$ and $I$ be a subinterval of $\R$. A weak solution $y(\cdot)$ of equation \eqref{1} is $n$ times strongly differentiable on $I$ iff
\begin{equation*}
y(t) \in D(A^n),\ t\in I,
\end{equation*}
in which case, 
\begin{equation*}
y^{(k)}(t)=A^ky(t),\ k=1,\dots,n,t\in I.
\end{equation*}
\end{prop}

\begin{proof}\quad
The statement immediately follows from the prior theorem and Proposition \ref{Prop} applied to
\[
y_+(t):=y(t)\ \text{and}\ y_-(t):=y(-t),\ t\ge 0,
\] 
for an arbitrary weak solution $y(\cdot)$ of equation \eqref{1}.
\end{proof}

\begin{rem}
Observe that, as well as for Proposition \ref{Prop}, for $n=1$, the subinterval $I$ can degenerate into a singleton.
\end{rem}

Inductively, we immediately obtain the following
analog of {\cite[Corollary $3.2$]{Markin2011}}: 

\begin{cor}[Infinite Differentiability of a Particular Weak Solution]\label{Cor}\ \\
Let $A$ be a scalar type spectral operator in a complex Banach space $(X,\|\cdot\|)$ and $I$ be a subinterval of $\R$. A weak solution $y(\cdot)$ of equation \eqref{1} is strongly infinite differentiable on $I$ ($y(\cdot)\in C^\infty(I,X)$) iff, for each $t\in I$, 
\begin{equation*}
y(t) \in C^\infty(A):=\bigcap_{n=1}^{\infty}D(A^n),
\end{equation*}
in which case
\begin{equation*}
y^{(n)}(t)=A^ny(t),\ n\in \N,t\in I.
\end{equation*}
\end{cor}

\section{Infinite Differentiability of Weak Solutions}

In this section, we characterize the strong infinite differentiability on $\R$ of all weak solutions of equation \eqref{1} with a scalar type spectral operator $A$ in a complex Banach space.

\begin{thm}[Infinite Differentiability of Weak Solutions]\label{real}\ \\
Let $A$ be a scalar type spectral operator in a complex Banach space $(X,\|\cdot\|)$ with spectral measure $E_A(\cdot)$. Every weak solution of equation \eqref{1} is strongly infinite differentiable on $\R$ iff there exist $b_+>0$ and $ b_->0$ such that the set $\sigma(A)\setminus {\mathscr L}_{b_-,b_+}$,
where
\begin{equation*}
{\mathscr L}_{b_-,b_+}:=\left\{\lambda \in \C\, \middle|\,
\Rep\lambda \le \min\left(0,-b_-\ln|\Imp\lambda|\right) 
\ \text{or}\ 
\Rep\lambda \ge \max\left(0,b_+\ln|\Imp\lambda|\right)\right\},
\end{equation*}
is bounded 
(see Fig. \ref{fig:graph2}).

\begin{figure}[h]
\centering
\includegraphics[height=2in]{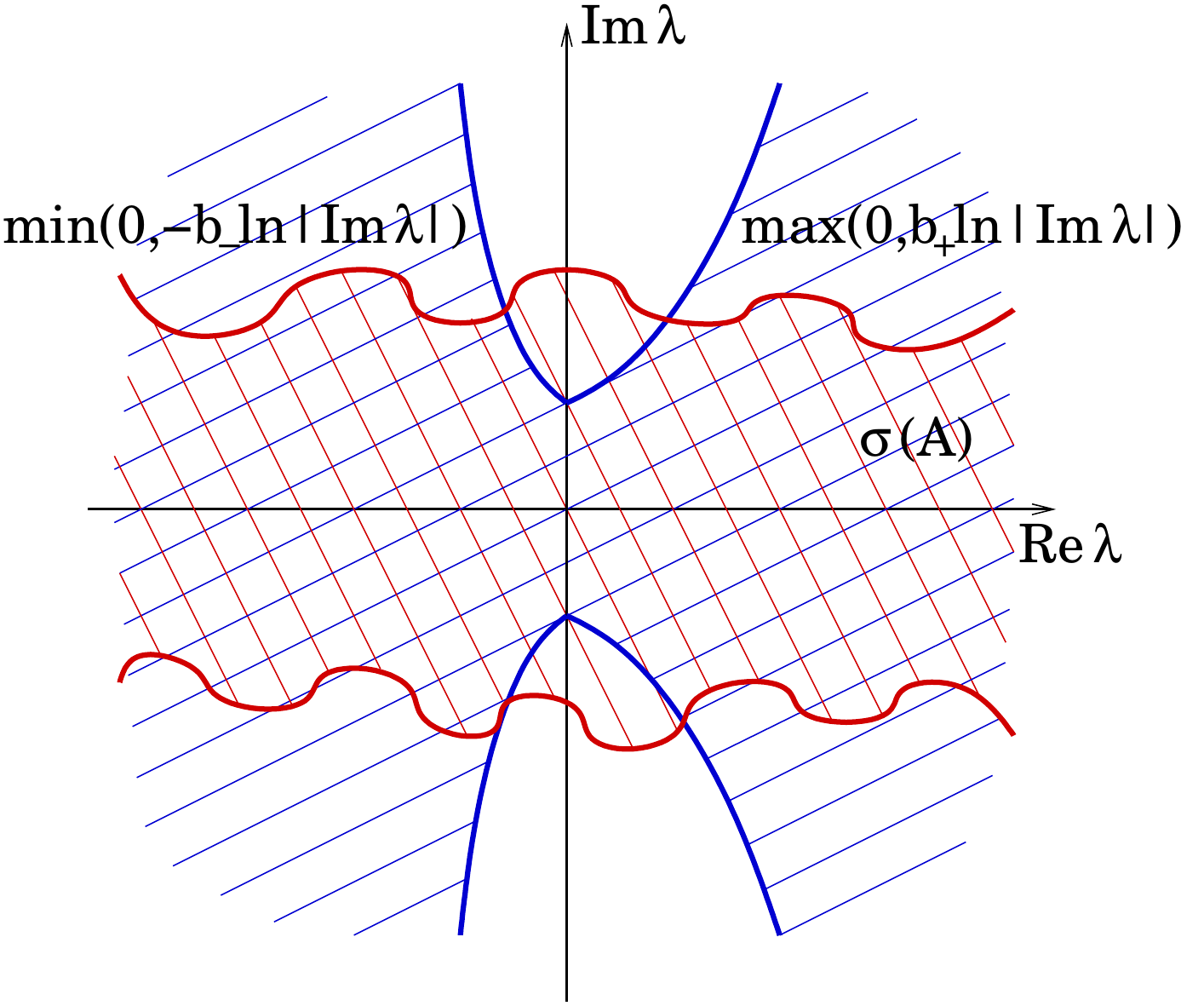}
\caption[]{}
\label{fig:graph2}
\end{figure}
\end{thm}

\begin{proof}

\textit{``If"} part.\quad Suppose that there exist
$b_+>0$ and $ b_->0$ such that the set $\sigma(A)\setminus {\mathscr L}_{b_-,b_+}$
is \textit{bounded} and let $y(\cdot)$ be an arbitrary weak solution of equation \eqref{1}. 

By Theorem \ref{GWS}, 
\begin{equation*}
y(t)=e^{tA}f,\ t\in \R,\ \text{with some}\
f \in \bigcap_{t\in \R}D(e^{tA}).
\end{equation*}

Our purpose is to show that $y(\cdot)\in C^\infty\left(\R,X\right)$, which, by Corollary \ref{Cor}, is attained by showing that, for each $t\in\R$,
\[
y(t)\in C^\infty(A):=\bigcap_{n=1}^\infty D(A^n).
\]

Let us proceed by proving that, for any $t\in\R$ and $m\in\N$
\[
y(t)\in D(A^m)
\] 
via Proposition \ref{prop}.

For any $t\in\R$, $m\in\N$ and an arbitrary $g^*\in X^*$,
\begin{multline}\label{first}
\int\limits_{\sigma(A)}|\lambda|^me^{t\Rep\lambda}\,dv(f,g^*,\lambda)
=\int\limits_{\sigma(A)\setminus{\mathscr L}_{b_-,b_+}}|\lambda|^me^{t\Rep\lambda}\,dv(f,g^*,\lambda)
\\
\shoveleft{
+\int\limits_{\left\{\lambda\in \sigma(A)\cap{\mathscr L}_{b_-,b_+}\,\middle|\,-1<\Rep\lambda<1 \right\}}|\lambda|^me^{t\Rep\lambda}\,dv(f,g^*,\lambda)
}\\
\shoveleft{
+\int\limits_{\left\{\lambda\in \sigma(A)\cap{\mathscr L}_{b_-,b_+}\,\middle|\,\Rep\lambda\ge 1 \right\}}|\lambda|^me^{t\Rep\lambda}\,dv(f,g^*,\lambda)
}\\
\hspace{1.2cm}
+\int\limits_{\left\{\lambda\in \sigma(A)\cap{\mathscr L}_{b_-,b_+}\,\middle|\,\Rep\lambda\le -1 \right\}}|\lambda|^me^{t\Rep\lambda}\,dv(f,g^*,\lambda)<\infty.
\hfill
\end{multline}

Indeed, 
\[
\int\limits_{\sigma(A)\setminus{\mathscr L}_{b_-,b_+}}|\lambda|^me^{t\Rep\lambda}\,dv(f,g^*,\lambda)<\infty
\]
and
\[
\int\limits_{\left\{\lambda\in \sigma(A)\cap{\mathscr L}_{b_-,b_+}\,\middle|\,-1<\Rep\lambda<1 \right\}}|\lambda|^me^{t\Rep\lambda}\,dv(f,g^*,\lambda)<\infty
\]
due to the boundedness of the sets
\[
\sigma(A)\setminus{\mathscr L}_{b_-,b_+}\ \text{and}\
\left\{\lambda\in \sigma(A)\cap{\mathscr L}_{b_-,b_+}\;\middle|\;-1<\Rep\lambda<1 \right\},
\]
the continuity of the integrated function on $\C$, and the finiteness of the measure $v(f,g^*,\cdot)$.

Further, for any $t\in\R$, $m\in\N$ and an arbitrary $g^*\in X^*$,
\begin{multline}\label{interm}
\int\limits_{\left\{\lambda\in \sigma(A)\cap{\mathscr L}_{b_-,b_+}\,\middle|\,\Rep\lambda\ge 1 \right\}}|\lambda|^me^{t\Rep\lambda}\,dv(f,g^*,\lambda)
\\
\shoveleft{
\le\int\limits_{\left\{\lambda\in \sigma(A)\cap{\mathscr L}_{b_-,b_+}\,\middle|\,\Rep\lambda\ge 1 \right\}}{\left[|\Rep\lambda|+|\Imp\lambda|\right]}^me^{t\Rep\lambda}\,dv(f,g^*,\lambda)
}\\
\hfill
\text{since, for $\lambda\in\sigma(A)\cap{\mathscr L}_{b_-,b_+}$ with $\Rep\lambda\ge 1$, $e^{b_+^{-1}\Rep\lambda}\ge |\Imp\lambda|$;}
\\
\shoveleft{
\le 
\int\limits_{\left\{\lambda\in \sigma(A)\cap{\mathscr L}_{b_-,b_+}\,\middle|\,\Rep\lambda\ge 1 \right\}}{\left[\Rep\lambda+e^{b_+^{-1}\Rep\lambda}\right]}^me^{t\Rep\lambda}\,dv(f,g^*,\lambda)
}\\
\hfill
\text{since, in view of $\Rep\lambda\ge 1$, $b_+e^{b_+^{-1}\Rep\lambda}\ge\Rep\lambda$;}
\\
\shoveleft{
\le 
\int\limits_{\left\{\lambda\in \sigma(A)\cap{\mathscr L}_{b_-,b_+}\,\middle|\,\Rep\lambda\ge 1 \right\}}{\left[b_+e^{b_+^{-1}\Rep\lambda}+e^{b_+^{-1}\Rep\lambda}\right]}^me^{t\Rep\lambda}\,dv(f,g^*,\lambda)
}\\
\shoveleft{
= [b_++1]^m\int\limits_{\left\{\lambda\in \sigma(A)\cap{\mathscr L}_{b_-,b_+}\,\middle|\,\Rep\lambda\ge 1 \right\}}e^{[mb_+^{-1}+t]\Rep\lambda}\,dv(f,g^*,\lambda)
}\\
\hfill
\text{since $f\in \bigcap\limits_{t\in\R}D(e^{tA})$, by Proposition \ref{prop};}
\\
\hspace{1.2cm}
<\infty. 
\hfill
\end{multline}

Finally, for any $t\in\R$, $m\in\N$ and an arbitrary $g^*\in X^*$,
\begin{multline}\label{interm2}
\int\limits_{\left\{\lambda\in \sigma(A)\cap{\mathscr L}_{b_-,b_+}\,\middle|\,\Rep\lambda\le -1 \right\}}|\lambda|^me^{t\Rep\lambda}\,dv(f,g^*,\lambda)
\\
\shoveleft{
\le\int\limits_{\left\{\lambda\in \sigma(A)\cap{\mathscr L}_{b_-,b_+}\,\middle|\,\Rep\lambda\le -1 \right\}}{\left[|\Rep\lambda|+|\Imp\lambda|\right]}^me^{t\Rep\lambda}\,dv(f,g^*,\lambda)
}\\
\hfill
\text{since, for $\lambda\in\sigma(A)\cap{\mathscr L}_{b_-,b_+}$ with $\Rep\lambda\le -1$, $e^{b_-^{-1}(-\Rep\lambda)}\ge |\Imp\lambda|$;}
\\
\shoveleft{
\le 
\int\limits_{\left\{\lambda\in \sigma(A)\cap{\mathscr L}_{b_-,b_+}\,\middle|\,\Rep\lambda\le -1 \right\}}{\left[-\Rep\lambda+e^{b_-^{-1}(-\Rep\lambda)}\right]}^me^{t\Rep\lambda}\,dv(f,g^*,\lambda)
}\\
\hfill
\text{since, in view of $-\Rep\lambda\ge 1$, $b_-e^{b_-^{-1}(-\Rep\lambda)}\ge-\Rep\lambda$;}
\\
\shoveleft{
\le 
\int\limits_{\left\{\lambda\in \sigma(A)\cap{\mathscr L}_{b_-,b_+}\,\middle|\,\Rep\lambda\le -1 \right\}}{\left[b_-e^{b_-^{-1}(-\Rep\lambda)}+e^{b_-^{-1}(-\Rep\lambda)}\right]}^me^{t\Rep\lambda}\,dv(f,g^*,\lambda)
}\\
\shoveleft{
= [b_-+1]^m\int\limits_{\left\{\lambda\in \sigma(A)\cap{\mathscr L}_{b_-,b_+}\,\middle|\,\Rep\lambda\le -1 \right\}}e^{[t-mb_-^{-1}]\Rep\lambda}\,dv(f,g^*,\lambda)
}\\
\hfill
\text{since $f\in \bigcap\limits_{t\in\R}D(e^{tA})$, by Proposition \ref{prop};}
\\
\hspace{1.2cm}
<\infty. 
\hfill
\end{multline}

Also, for any $t\in\R$, $m\in\N$ and an arbitrary $n\in\N$,
\begin{multline}\label{second}
\sup_{\{g^*\in X^*\,|\,\|g^*\|=1\}}
\int\limits_{\left\{\lambda\in\sigma(A)\,\middle|\,|\lambda|^me^{t\Rep\lambda}>n\right\}}
|\lambda|^me^{t\Rep\lambda}\,dv(f,g^*,\lambda)
\\
\shoveleft{
\le \sup_{\{g^*\in X^*\,|\,\|g^*\|=1\}}
\int\limits_{\left\{\lambda\in\sigma(A)
\setminus{\mathscr L}_{b_-,b_+}\,\middle|\,|\lambda|^me^{t\Rep\lambda}>n\right\}}|\lambda|^me^{t\Rep\lambda}\,dv(f,g^*,\lambda)
}\\
\shoveleft{
+ \sup_{\{g^*\in X^*\,|\,\|g^*\|=1\}}
\int\limits_{\left\{\lambda\in\sigma(A)\cap{\mathscr L}_{b_-,b_+}\,\middle|\,-1<\Rep\lambda<1,\, |\lambda|^me^{t\Rep\lambda}>n\right\}}|\lambda|^me^{t\Rep\lambda}\,dv(f,g^*,\lambda)
}\\
\shoveleft{
+ \sup_{\{g^*\in X^*\,|\,\|g^*\|=1\}}
\int\limits_{\left\{\lambda\in\sigma(A)\cap{\mathscr L}_{b_-,b_+}\,\middle|\,\Rep\lambda\ge 1,\, |\lambda|^me^{t\Rep\lambda}>n\right\}}|\lambda|^me^{t\Rep\lambda}\,dv(f,g^*,\lambda)
}\\
\shoveleft{
+ \sup_{\{g^*\in X^*\,|\,\|g^*\|=1\}}
\int\limits_{\left\{\lambda\in\sigma(A)\cap{\mathscr L}_{b_-,b_+}\,\middle|\,\Rep\lambda\le -1,\, |\lambda|^me^{t\Rep\lambda}>n\right\}}|\lambda|^me^{t\Rep\lambda}\,dv(f,g^*,\lambda)
}\\
\hspace{1.2cm}
\to 0,\ n\to\infty.
\hfill
\end{multline}

Indeed, since, due to the boundedness of the sets
\[
\sigma(A)\setminus{\mathscr L}_{b_-,b_+}\ \text{and}\
\left\{\lambda\in\sigma(A)\cap{\mathscr L}_{b_-,b_+}\,\middle|\,-1<\Rep\lambda<1\right\}
\]
and the continuity of the integrated function on $\C$,
the sets
\[
\left\{\lambda\in\sigma(A)\setminus{\mathscr L}_{b_-,b_+}\,\middle|\,|\lambda|^me^{t\Rep\lambda}>n\right\}
\]
and 
\[
\left\{\lambda\in\sigma(A)\cap{\mathscr L}_{b_-,b_+}\,\middle|\,-1<\Rep\lambda<1,\, |\lambda|^me^{t\Rep\lambda}>n\right\}
\]
are \textit{empty} for all sufficiently large $n\in \N$,
we immediately infer that, for any $t\in\R$ and $m\in\N$,
\[
\lim_{n\to\infty}\sup_{\{g^*\in X^*\,|\,\|g^*\|=1\}}
\int\limits_{\left\{\lambda\in\sigma(A)
\setminus{\mathscr L}_{b_-,b_+}\,\middle|\,|\lambda|^me^{t\Rep\lambda}>n\right\}}|\lambda|^me^{t\Rep\lambda}\,dv(f,g^*,\lambda)=0
\]
and
\[
\lim_{n\to\infty}\sup_{\{g^*\in X^*\,|\,\|g^*\|=1\}}
\int\limits_{\left\{\lambda\in\sigma(A)\cap{\mathscr L}_{b_-,b_+}\,\middle|\,-1<\Rep\lambda<1,\, |\lambda|^me^{t\Rep\lambda}>n\right\}}|\lambda|^me^{t\Rep\lambda}\,dv(f,g^*,\lambda)
=0.
\]

Further, for any $t\in\R$, $m\in\N$ and an arbitrary $n\in\N$,
\begin{multline*}
\sup_{\{g^*\in X^*\,|\,\|g^*\|=1\}}
\int\limits_{\left\{\lambda\in\sigma(A)\cap{\mathscr L}_{b_-,b_+}\,\middle|\,\Rep\lambda\ge 1,\, |\lambda|^me^{t\Rep\lambda}>n\right\}}|\lambda|^me^{t\Rep\lambda}\,dv(f,g^*,\lambda)
\\
\hfill
\text{as in \eqref{interm};}
\\
\shoveleft{
\le \sup_{\{g^*\in X^*\,|\,\|g^*\|=1\}}
[b_++1]^m\int\limits_{\left\{\lambda\in\sigma(A)\cap{\mathscr L}_{b_-,b_+}\,\middle|\,\Rep\lambda\ge 1,\, |\lambda|^me^{t\Rep\lambda}>n\right\}}e^{[mb_+^{-1}+t]\Rep\lambda}\,dv(f,g^*,\lambda)
}\\
\hfill
\text{since $f\in \bigcap\limits_{t\in\R}D(e^{tA})$, by \eqref{cond(ii)};}
\\
\shoveleft{
\le \sup_{\{g^*\in X^*\,|\,\|g^*\|=1\}}
}\\
\shoveleft{
[b_++1]^m4M\left\|E_A\left(\left\{\lambda\in\sigma(A)\cap{\mathscr L}_{b_-,b_+}\,\middle|\,\Rep\lambda\ge 1,\, |\lambda|^me^{t\Rep\lambda}>n\right\}\right)
e^{[mb_+^{-1}+t]A}f\right\|\|g^*\|
}\\
\shoveleft{
\le [b_++1]^m4M\left\|E_A\left(\left\{\lambda\in\sigma(A)\cap{\mathscr L}_{b_-,b_+}\,\middle|\,\Rep\lambda\ge 1,\, |\lambda|^me^{t\Rep\lambda}>n\right\}\right)
e^{[mb_+^{-1}+t]A}f\right\|
}\\
\hfill
\text{by the strong continuity of the {\it s.m.};}
\\
\ \
\to [b_++1]^m4M\left\|E_A\left(\emptyset\right)e^{[mb_+^{-1}+t]A}f\right\|=0,\ n\to\infty.
\hfill
\end{multline*}

Finally, for any $t\in\R$, $m\in\N$ and an arbitrary $n\in\N$,
\begin{multline*}
\sup_{\{g^*\in X^*\,|\,\|g^*\|=1\}}
\int\limits_{\left\{\lambda\in\sigma(A)\cap{\mathscr L}_{b_-,b_+}\,\middle|\,\Rep\lambda\le -1,\, |\lambda|^me^{t\Rep\lambda}>n\right\}}|\lambda|^me^{t\Rep\lambda}\,dv(f,g^*,\lambda)
\\
\hfill
\text{as in \eqref{interm2};}
\\
\shoveleft{
\le \sup_{\{g^*\in X^*\,|\,\|g^*\|=1\}}
[b_-+1]^m\int\limits_{\left\{\lambda\in\sigma(A)\cap{\mathscr L}_{b_-,b_+}\,\middle|\,\Rep\lambda\le -1,\, |\lambda|^me^{t\Rep\lambda}>n\right\}}e^{[t-mb_-^{-1}]\Rep\lambda}\,dv(f,g^*,\lambda)
}\\
\hfill
\text{since $f\in \bigcap\limits_{t\in\R}D(e^{tA})$, by \eqref{cond(ii)};}
\\
\shoveleft{
\le \sup_{\{g^*\in X^*\,|\,\|g^*\|=1\}}
}\\
\shoveleft{
[b_-+1]^m4M\left\|E_A\left(\left\{\lambda\in\sigma(A)\cap{\mathscr L}_{b_-,b_+}\,\middle|\,\Rep\lambda\le -1,\, |\lambda|^me^{t\Rep\lambda}>n\right\}\right)
e^{[t-mb_-^{-1}]A}f\right\|\|g^*\|
}\\
\shoveleft{
\le [b_-+1]^m4M\left\|E_A\left(\left\{\lambda\in\sigma(A)\cap{\mathscr L}_{b_-,b_+}\,\middle|\,\Rep\lambda\le -1,\, |\lambda|^me^{t\Rep\lambda}>n\right\}\right)
e^{[t-mb_-^{-1}]A}f\right\|
}\\
\hfill
\text{by the strong continuity of the {\it s.m.};}
\\
\ \
\to [b_++1]^m4M\left\|E_A\left(\emptyset\right)e^{[t-mb_-^{-1}]A}f\right\|=0,\ n\to\infty.
\hfill
\end{multline*}

By Proposition \ref{prop} and the properties of the \textit{o.c.} (see {\cite[Theorem XVIII.$2.11$ (f)]{Dun-SchIII}}), \eqref{first} and \eqref{second} jointly imply that, for any $t\in\R$ and $m\in\N$,
\[
f\in D(A^me^{tA}),
\]
which further implies that, for each $t\in\R$, 
\begin{equation*}
y(t)=e^{tA}f\in \bigcap_{n=1}^\infty D(A^n)
=:C^\infty(A).
\end{equation*}

Whence, by Corollary \ref{Cor}, we infer that
\begin{equation*}
y(\cdot) \in C^\infty\left(\R,X\right),
\end{equation*}
which completes the proof of the \textit{``if"} part.

\medskip
\textit{``Only if"} part.\quad Let us prove this part {\it by contrapositive} assuming that, for any $b_+>0$ and $b_->0$, the set 
$\sigma(A)\setminus {\mathscr L}_{b_-,b_+}$ is \textit{unbounded}. In particular, this means that, for any $n\in \N$, unbounded is the set
\begin{equation*}
\sigma(A)\setminus {\mathscr L}_{(2n)^{-1},(2n)^{-1}}=
\left\{\lambda \in \sigma(A)\,\middle| 
-(2n)^{-1}\ln|\Imp\lambda|<\Rep\lambda < (2n)^{-1}\ln|\Imp\lambda|\right\}.
\end{equation*} 

Hence, we can choose a sequence of points $\left\{\lambda_n\right\}_{n=1}^\infty$ 
in the complex plane as follows:
\begin{equation*}
\begin{split}
&\lambda_n \in \sigma(A),\ n\in \N,\\
&-(2n)^{-1}\ln|\Imp\lambda_n|<\Rep\lambda_n < (2n)^{-1}\ln|\Imp\lambda_n|,\ n\in \N,\\
&\lambda_0:=0,\ |\lambda_n|>\max\left[n^4,|\lambda_{n-1}|\right],\ n\in \N.\\
\end{split}
\end{equation*}

The latter implies, in particular, that the points $\lambda_n$, $n\in\N$, are \textit{distinct} ($\lambda_i \neq \lambda_j$, $i\neq j$).

Since, for each $n\in \N$, the set
\begin{equation*}
\left\{ \lambda \in {\mathbb C}\,\middle|\, 
-(2n)^{-1}\ln|\Imp\lambda|<\Rep\lambda < (2n)^{-1}\ln|\Imp\lambda|,\
|\lambda|>\max\bigl[n^4,|\lambda_{n-1}|\bigr]\right\}
\end{equation*}
is {\it open} in $\C$, along with the point $\lambda_n$, it contains an {\it open disk}
\begin{equation*}
\Delta_n:=\left\{\lambda \in \C\, \middle|\,|\lambda-\lambda_n|<\varepsilon_n \right\}
\end{equation*} 
centered at $\lambda_n$ of some radius $\varepsilon_n>0$, i.e., for each $\lambda \in \Delta_n$,
\begin{equation}\label{disks1}
-(2n)^{-1}\ln|\Imp\lambda|<\Rep\lambda < (2n)^{-1}\ln|\Imp\lambda|\ \text{and}\ |\lambda|>\max\bigl[n^4,|\lambda_{n-1}|\bigr].
\end{equation}

Furthermore, we can regard the radii of the disks to be small enough so that
\begin{equation}\label{radii1}
\begin{split}
&0<\varepsilon_n<\dfrac{1}{n},\ n\in\N,\ \text{and}\\
&\Delta_i \cap \Delta_j=\emptyset,\ i\neq j
\quad \text{(i.e., the disks are {\it pairwise disjoint})}.
\end{split}
\end{equation}

Whence, by the properties of the {\it s.m.}, 
\begin{equation*}
E_A(\Delta_i)E_A(\Delta_j)=0,\ i\neq j,
\end{equation*}
where $0$ stands for the \textit{zero operator} on $X$.

Observe also, that the subspaces $E_A(\Delta_n)X$, $n\in \N$, are \textit{nontrivial} since
\[
\Delta_n \cap \sigma(A)\neq \emptyset,\ n\in\N,
\]
with $\Delta_n$ being an \textit{open set} in $\C$. 

By choosing a unit vector $e_n\in E_A(\Delta_n)X$ for each $n\in\N$, we obtain a sequence 
$\left\{e_n\right\}_{n=1}^\infty$ in $X$ such that
\begin{equation}\label{ortho1}
\|e_n\|=1,\ n\in\N,\ \text{and}\ E_A(\Delta_i)e_j=\delta_{ij}e_j,\ i,j\in\N,
\end{equation}
where $\delta_{ij}$ is the \textit{Kronecker delta}.

As is easily seen, \eqref{ortho1} implies that the vectors $e_n$, $n\in \N$, are \textit{linearly independent}.

Furthermore, there is an $\varepsilon>0$ such that
\begin{equation}\label{dist1}
d_n:=\dist\left(e_n,\spa\left(\left\{e_i\,|\,i\in\N,\ i\neq n\right\}\right)\right)\ge\varepsilon,\ n\in\N.
\end{equation}

Indeed, the opposite implies the existence of a subsequence $\left\{d_{n(k)}\right\}_{k=1}^\infty$ such that
\begin{equation*}
d_{n(k)}\to 0,\ k\to\infty.
\end{equation*}

Then, by selecting a vector
\[
f_{n(k)}\in 
\spa\left(\left\{e_i\,|\,i\in\N,\ i\neq n(k)\right\}\right),\ k\in\N,
\] 
such that 
\[
\|e_{n(k)}-f_{n(k)}\|<d_{n(k)}+1/k,\ k\in\N,
\]
we arrive at
\begin{multline*}
1=\|e_{n(k)}\|
\hfill
\text{since, by \eqref{ortho1}, 
$E_A(\Delta_{n(k)})f_{n(k)}=0$;}
\\
\shoveleft{
=\|E_A(\Delta_{n(k)})(e_{n(k)}-f_{n(k)})\|\
\le \|E_A(\Delta_{n(k)})\|\|e_{n(k)}-f_{n(k)}\|
\hfill
\text{by \eqref{bounded};}
}\\
\ \
\le M\|e_{n(k)}-f_{n(k)}\|\le M\left[d_{n(k)}+1/k\right]
\to 0,\ k\to\infty,
\hfill
\end{multline*}
which is a \textit{contradiction} proving \eqref{dist1}. 

As follows from the {\it Hahn-Banach Theorem}, for any $n\in\N$, there is an $e^*_n\in X^*$ such that 
\begin{equation}\label{H-B1}
\|e_n^*\|=1,\ n\in\N,\ \text{and}\ \langle e_i,e_j^*\rangle=\delta_{ij}d_i,\ i,j\in\N.
\end{equation}

Let us consider separately the two possibilities concerning the sequence of the real parts $\{\Rep\lambda_n\}_{n=1}^\infty$: its being \textit{bounded} or \textit{unbounded}. 

First, suppose that the sequence $\{\Rep\lambda_n\}_{n=1}^\infty$ is \textit{bounded}, i.e., there is such an $\omega>0$ that
\begin{equation}\label{bounded1}
|\Rep\lambda_n| \le \omega,\ n\in\N,
\end{equation}
and consider the element
\begin{equation*}
f:=\sum_{k=1}^\infty k^{-2}e_k\in X,
\end{equation*}
which is well defined since $\left\{k^{-2}\right\}_{k=1}^\infty\in l_1$ ($l_1$ is the space of absolutely summable sequences) and $\|e_k\|=1$, $k\in\N$ (see \eqref{ortho1}).

In view of \eqref{ortho1}, by the properties of the \textit{s.m.},
\begin{equation}\label{vectors1}
E_A(\cup_{k=1}^\infty\Delta_k)f=f\ \text{and}\ E_A(\Delta_k)f=k^{-2}e_k,\ k\in\N.
\end{equation}

For any $t\ge 0$ and an arbitrary $g^*\in X^*$,
\begin{multline}\label{first1}
\int\limits_{\sigma(A)}e^{t\Rep\lambda}\,dv(f,g^*,\lambda)
\hfill \text{by \eqref{vectors1};}
\\
\shoveleft{
=\int\limits_{\sigma(A)} e^{t\Rep\lambda}\,d v(E_A(\cup_{k=1}^\infty \Delta_k)f,g^*,\lambda)
\hfill
\text{by \eqref{decompose};}
}\\
\shoveleft{
=\sum_{k=1}^\infty\int\limits_{\sigma(A)\cap\Delta_k}e^{t\Rep\lambda}\,dv(E_A(\Delta_k)f,g^*,\lambda)
\hfill 
\text{by \eqref{vectors1};}
}\\
\shoveleft{
=\sum_{k=1}^\infty k^{-2}\int\limits_{\sigma(A)\cap\Delta_k}e^{t\Rep\lambda}\,dv(e_k,g^*,\lambda)
}\\
\hfill
\text{since, for $\lambda\in \Delta_k$, by \eqref{bounded1} and \eqref{radii1},}\ 
\Rep\lambda=\Rep\lambda_k+(\Rep\lambda-\Rep\lambda_k)
\\
\hfill
\le \Rep\lambda_k+|\lambda-\lambda_k|\le \omega+\varepsilon_k\le \omega+1;
\\
\shoveleft{
\le e^{t(\omega+1)}\sum_{k=1}^\infty k^{-2}\int\limits_{\sigma(A)\cap\Delta_k}1\,dv(e_k,g^*,\lambda)
= e^{t(\omega+1)}\sum_{k=1}^\infty k^{-2}v(e_k,g^*,\Delta_k)
}\\
\hfill
\text{by \eqref{tv};}
\\
\hspace{1.2cm}
\le e^{t(\omega+1)}\sum_{k=1}^\infty k^{-2}4M\|e_k\|\|g^*\|
= 4Me^{t(\omega+1)}\|g^*\|\sum_{k=1}^\infty k^{-2}<\infty.
\hfill
\end{multline} 

Also, for any $t<0$ and an arbitrary $g^*\in X^*$,
\begin{multline}\label{ffirst1}
\int\limits_{\sigma(A)}e^{t\Rep\lambda}\,dv(f,g^*,\lambda)
\hfill \text{by \eqref{vectors1};}
\\
\shoveleft{
=\int\limits_{\sigma(A)} e^{t\Rep\lambda}\,d v(E_A(\cup_{k=1}^\infty \Delta_k)f,g^*,\lambda)
\hfill
\text{by \eqref{decompose};}
}\\
\shoveleft{
=\sum_{k=1}^\infty\int\limits_{\sigma(A)\cap\Delta_k}e^{t\Rep\lambda}\,dv(E_A(\Delta_k)f,g^*,\lambda)
\hfill 
\text{by \eqref{vectors1};}
}\\
\shoveleft{
=\sum_{k=1}^\infty k^{-2}\int\limits_{\sigma(A)\cap\Delta_k}e^{t\Rep\lambda}\,dv(e_k,g^*,\lambda)
}\\
\hfill
\text{since, for $\lambda\in \Delta_k$, by \eqref{bounded1} and \eqref{radii1},}\ 
\Rep\lambda=\Rep\lambda_k-(\Rep\lambda_k-\Rep\lambda)
\\
\hfill
\ge \Rep\lambda_k-|\Rep\lambda_k-\Rep\lambda|\ge -\omega-\varepsilon_k\ge -\omega-1;
\\
\shoveleft{
\le e^{-t(\omega+1)}\sum_{k=1}^\infty k^{-2}\int\limits_{\sigma(A)\cap\Delta_k}1\,dv(e_k,g^*,\lambda)
= e^{-t(\omega+1)}\sum_{k=1}^\infty k^{-2}v(e_k,g^*,\Delta_k)
}\\
\hfill
\text{by \eqref{tv};}
\\
\hspace{1.2cm}
\le e^{-t(\omega+1)}\sum_{k=1}^\infty k^{-2}4M\|e_k\|\|g^*\|
= 4Me^{-t(\omega+1)}\|g^*\|\sum_{k=1}^\infty k^{-2}<\infty.
\hfill
\end{multline} 

Similarly, to \eqref{first1} for any $t\ge 0$ and an arbitrary $n\in\N$,
\begin{multline}\label{second1}
\sup_{\{g^*\in X^*\,|\,\|g^*\|=1\}}
\int\limits_{\left\{\lambda\in\sigma(A)\,\middle|\,e^{t\Rep\lambda}>n\right\}} 
e^{t\Rep\lambda}\,dv(f,g^*,\lambda)
\\
\shoveleft{
\le 
\sup_{\{g^*\in X^*\,|\,\|g^*\|=1\}}e^{t(\omega+1)}\sum_{k=1}^\infty k^{-2}
\int\limits_{\left\{\lambda\in\sigma(A)\,\middle|\,e^{t\Rep\lambda}>n\right\}\cap \Delta_k}1\,dv(e_k,g^*,\lambda) 
}\\
\hfill \text{by \eqref{vectors1};}
\\
\shoveleft{
=e^{t(\omega+1)}\sup_{\{g^*\in X^*\,|\,\|g^*\|=1\}}\sum_{k=1}^\infty 
\int\limits_{\left\{\lambda\in\sigma(A)\,\middle|\,e^{t\Rep\lambda}>n\right\}\cap \Delta_k}1\,dv(E_A(\Delta_k)f,g^*,\lambda) 
}\\
\hfill \text{by \eqref{decompose};}
\\
\shoveleft{
= e^{t(\omega+1)}\sup_{\{g^*\in X^*\,|\,\|g^*\|=1\}}
\int\limits_{\{\lambda\in\sigma(A)\,|\,e^{t\Rep\lambda}>n\}}1\,dv(E_A(\cup_{k=1}^\infty\Delta_k)f,g^*,\lambda)
}\\
\hfill \text{by \eqref{vectors1};}
\\
\shoveleft{
= e^{t(\omega+1)}\sup_{\{g^*\in X^*\,|\,\|g^*\|=1\}}
\int\limits_{\{\lambda\in\sigma(A)\,|\,e^{t\Rep\lambda}>n\}}1\,dv(f,g^*,\lambda)
\hfill
\text{by \eqref{cond(ii)};}
}\\
\shoveleft{
\le e^{t(\omega+1)}\sup_{\{g^*\in X^*\,|\,\|g^*\|=1\}}4M\left\|E_A\left(\left\{\lambda\in\sigma(A)\,\middle|\,e^{t\Rep\lambda}>n\right\}\right)f\right\|\|g^*\|
}\\
\shoveleft{
\le 4Me^{t(\omega+1)}\left\|E_A\left(\left\{\lambda\in\sigma(A)\,\middle|\,e^{t\Rep\lambda}>n\right\}\right)f\right\|
}\\
\hfill
\text{by the strong continuity of the {\it s.m.};}
\\
\hspace{1.2cm}
\to 4Me^{t(\omega+1)}\left\|E_A\left(\emptyset\right)f\right\|=0,\ n\to\infty.
\hfill
\end{multline}

Similarly, to \eqref{ffirst1} for any $t<0$ and an arbitrary $n\in\N$,
\begin{multline}\label{ssecond1}
\sup_{\{g^*\in X^*\,|\,\|g^*\|=1\}}
\int\limits_{\left\{\lambda\in\sigma(A)\,\middle|\,e^{t\Rep\lambda}>n\right\}} 
e^{t\Rep\lambda}\,dv(f,g^*,\lambda)
\\
\shoveleft{
\le 
\sup_{\{g^*\in X^*\,|\,\|g^*\|=1\}}e^{-t(\omega+1)}\sum_{k=1}^\infty k^{-2}
\int\limits_{\left\{\lambda\in\sigma(A)\,\middle|\,e^{t\Rep\lambda}>n\right\}\cap \Delta_k}1\,dv(e_k,g^*,\lambda) 
}\\
\hfill \text{by \eqref{vectors1};}
\\
\shoveleft{
=e^{-t(\omega+1)}\sup_{\{g^*\in X^*\,|\,\|g^*\|=1\}}\sum_{k=1}^\infty 
\int\limits_{\left\{\lambda\in\sigma(A)\,\middle|\,e^{t\Rep\lambda}>n\right\}\cap \Delta_k}1\,dv(E_A(\Delta_k)f,g^*,\lambda) 
}\\
\hfill \text{by \eqref{decompose};}
\\
\shoveleft{
= e^{-t(\omega+1)}\sup_{\{g^*\in X^*\,|\,\|g^*\|=1\}}
\int\limits_{\{\lambda\in\sigma(A)\,|\,e^{t\Rep\lambda}>n\}}1\,dv(E_A(\cup_{k=1}^\infty\Delta_k)f,g^*,\lambda)
}\\
\hfill \text{by \eqref{vectors1};}
\\
\shoveleft{
= e^{-t(\omega+1)}\sup_{\{g^*\in X^*\,|\,\|g^*\|=1\}}
\int\limits_{\{\lambda\in\sigma(A)\,|\,e^{t\Rep\lambda}>n\}}1\,dv(f,g^*,\lambda)
\hfill
\text{by \eqref{cond(ii)};}
}\\
\shoveleft{
\le e^{-t(\omega+1)}\sup_{\{g^*\in X^*\,|\,\|g^*\|=1\}}4M\left\|E_A\left(\left\{\lambda\in\sigma(A)\,\middle|\,e^{t\Rep\lambda}>n\right\}\right)f\right\|\|g^*\|
}\\
\shoveleft{
\le 4Me^{-t(\omega+1)}\left\|E_A\left(\left\{\lambda\in\sigma(A)\,\middle|\,e^{t\Rep\lambda}>n\right\}\right)f\right\|
}\\
\hfill
\text{by the strong continuity of the {\it s.m.};}
\\
\hspace{1.2cm}
\to 4Me^{-t(\omega+1)}\left\|E_A\left(\emptyset\right)f\right\|=0,\ n\to\infty.
\hfill
\end{multline}

By Proposition \ref{prop}, \eqref{first1}, \eqref{ffirst1}, \eqref{second1}, and \eqref{ssecond1} jointly imply that 
\[
f\in \bigcap\limits_{t\in\R}D(e^{tA}),
\]
and hence, by Theorem \ref{GWS},
\[
y(t):=e^{tA}f,\ t\in\R,
\]
is a weak solution of equation \eqref{1}.

Let
\begin{equation}\label{functional1}
h^*:=\sum_{k=1}^\infty k^{-2}e_k^*\in X^*,
\end{equation}
the functional being well defined since $\{k^{-2}\}_{k=1}^\infty\in l_1$ and $\|e_k^*\|=1$, $k\in\N$ (see \eqref{H-B1}).

In view of \eqref{H-B1} and \eqref{dist1}, we have:
\begin{equation}\label{funct-dist1}
\langle e_n,h^*\rangle=\langle e_k,k^{-2}e_k^*\rangle=d_k k^{-2}\ge \varepsilon k^{-2},\ k\in\N.
\end{equation}

Hence,
\begin{multline}\label{notin1}
\int\limits_{\sigma(A)}|\lambda|\,dv(f,h^*,\lambda)
\hfill
\text{by \eqref{decompose} as in \eqref{first1};}
\\
\shoveleft{
=\sum_{k=1}^\infty k^{-2}\int\limits_{\sigma(A)\cap \Delta_k}|\lambda|\,dv(e_k,h^*,\lambda)
}\\
\hfill
\text{since, for $\lambda\in \Delta_k$, by
\eqref{disks1}, $|\lambda|\ge k^4$;}
\\
\shoveleft{
\ge \sum_{k=1}^\infty k^{-2}k^4 v(e_k,h^*,\Delta_k)\ge\sum_{k=1}^\infty k^2|\langle E_A(\Delta_k)e_k,h^*\rangle|
}\\
\hfill
\text{by \eqref{ortho1} and \eqref{funct-dist1};}
\\
\hspace{1.2cm}
\ge \sum_{k=1}^\infty k^2 \varepsilon k^{-2}=\infty.
\hfill
\end{multline} 

By Proposition \ref{prop}, \eqref{notin1} implies that
\[
y(0)=f\notin D(A),
\]
which, by Proposition \ref{particular} ($n=1$, $I=\{0\}$) further implies that the weak solution $y(t)=e^{tA}f$, $t\in\R$, of equation \eqref{1} is not strongly differentiable at $0$. 

Now, suppose that the sequence $\{\Rep\lambda_n\}_{n=1}^\infty$
is \textit{unbounded}. 

Therefore, there is a subsequence $\{\Rep\lambda_{n(k)}\}_{k=1}^\infty$ such that
\[
\Rep\lambda_{n(k)}\to \infty \ \text{or}\ \Rep\lambda_{n(k)}\to -\infty,\ k\to \infty.
\]
Let us consider separately each of the two cases.

First, suppose that 
\[
\Rep\lambda_{n(k)}\to \infty,\ k\to \infty
\] 
Then, without loss of generality, we can regard that
\begin{equation}\label{infinity}
\Rep\lambda_{n(k)} \ge k,\ k\in\N.
\end{equation}

Consider the elements
\begin{equation*}
f:=\sum_{k=1}^\infty e^{-n(k)\Rep\lambda_{n(k)}}e_{n(k)}\in X
\ \text{and}\ h:=\sum_{k=1}^\infty e^{-\frac{n(k)}{2}\Rep\lambda_{n(k)}}e_{n(k)}\in X,
\end{equation*}
well defined since, by \eqref{infinity},
\[
\left\{e^{-n(k)\Rep\lambda_{n(k)}}\right\}_{k=1}^\infty,
\left\{e^{-\frac{n(k)}{2}\Rep\lambda_{n(k)}}\right\}_{k=1}^\infty
\in l_1
\]
and $\|e_{n(k)}\|=1$, $k\in\N$ (see \eqref{ortho1}).

By \eqref{ortho1},
\begin{equation}\label{subvectors1}
E_A(\cup_{k=1}^\infty\Delta_{n(k)})f=f\ \text{and}\
E_A(\Delta_{n(k)})f=e^{-n(k)\Rep\lambda_{n(k)}}e_{n(k)},\
k\in\N,
\end{equation}
and
\begin{equation}\label{subvectors12}
E_A(\cup_{k=1}^\infty\Delta_{n(k)})h=h\ \text{and}\
E_A(\Delta_{n(k)})h=e^{-\frac{n(k)}{2}\Rep\lambda_{n(k)}}e_{n(k)},\ k\in\N.
\end{equation}

For any $t\ge 0$ and an arbitrary $g^*\in X^*$, 
\begin{multline}\label{first2}
\int\limits_{\sigma(A)}e^{t\Rep\lambda}\,dv(f,g^*,\lambda)
\hfill
\text{by \eqref{decompose} as in \eqref{first1};}
\\
\shoveleft{
=\sum_{k=1}^\infty e^{-n(k)\Rep\lambda_{n(k)}}\int\limits_{\sigma(A)\cap\Delta_{n(k)}}e^{t\Rep\lambda}\,dv(e_{n(k)},g^*,\lambda)
}\\
\hfill
\text{since, for $\lambda\in \Delta_{n(k)}$, by \eqref{radii1},}\ \Rep\lambda
=\Rep\lambda_{n(k)}+(\Rep\lambda-\Rep\lambda_{n(k)})
\\
\hfill
\le \Rep\lambda_{n(k)}+|\lambda-\lambda_{n(k)}|\le \Rep\lambda_{n(k)}+1;
\\
\shoveleft{
\le \sum_{k=1}^\infty e^{-n(k)\Rep\lambda_{n(k)}}
e^{t(\Rep\lambda_{n(k)}+1)}
\int\limits_{\sigma(A)\cap\Delta_{n(k)}}1\,dv(e_{n(k)},g^*,\lambda)
}\\
\shoveleft{
= e^t\sum_{k=1}^\infty e^{-[n(k)-t]\Rep\lambda_{n(k)}}v(e_{n(k)},g^*,\Delta_{n(k)})
\hfill
\text{by \eqref{tv};}
}\\
\shoveleft{
\le e^t\sum_{k=1}^\infty e^{-[n(k)-t]\Rep\lambda_{n(k)}}4M\|e_{n(k)}\|\|g^*\|
= 4Me^t\|g^*\|\sum_{k=1}^\infty e^{-[n(k)-t]\Rep\lambda_{n(k)}}
}\\
\hspace{1.2cm}
<\infty.
\hfill
\end{multline}

Indeed, for all $k\in \N$ sufficiently large so that
\[
n(k)\ge t+1,
\]
in view of \eqref{infinity}, 
\[
e^{-[n(k)-t]\Rep\lambda_{n(k)}}\le e^{-k}.
\]

For any $t<0$ and an arbitrary $g^*\in X^*$, 
\begin{multline}\label{ffirst2}
\int\limits_{\sigma(A)}e^{t\Rep\lambda}\,dv(f,g^*,\lambda)
\hfill
\text{by \eqref{decompose} as in \eqref{first1};}
\\
\shoveleft{
=\sum_{k=1}^\infty e^{-n(k)\Rep\lambda_{n(k)}}\int\limits_{\sigma(A)\cap\Delta_{n(k)}}e^{t\Rep\lambda}\,dv(e_{n(k)},g^*,\lambda)
}\\
\hfill
\text{since, for $\lambda\in \Delta_{n(k)}$, by \eqref{radii1},}\ \Rep\lambda
=\Rep\lambda_{n(k)}-(\Rep\lambda_{n(k)}-\Rep\lambda)
\\
\hfill
\ge \Rep\lambda_{n(k)}-|\Rep\lambda_{n(k)}-\Rep\lambda|\ge \Rep\lambda_{n(k)}-1;
\\
\shoveleft{
\le \sum_{k=1}^\infty e^{-n(k)\Rep\lambda_{n(k)}}
e^{t(\Rep\lambda_{n(k)}-1)}
\int\limits_{\sigma(A)\cap\Delta_{n(k)}}1\,dv(e_{n(k)},g^*,\lambda)
}\\
\shoveleft{
= e^{-t}\sum_{k=1}^\infty e^{-[n(k)-t]\Rep\lambda_{n(k)}}v(e_{n(k)},g^*,\Delta_{n(k)})
\hfill
\text{by \eqref{tv};}
}\\
\shoveleft{
\le e^{-t}\sum_{k=1}^\infty e^{-[n(k)-t]\Rep\lambda_{n(k)}}4M\|e_{n(k)}\|\|g^*\|
= 4Me^{-t}\|g^*\|\sum_{k=1}^\infty e^{-[n(k)-t]\Rep\lambda_{n(k)}}
}\\
\hspace{1.2cm}
<\infty.
\hfill
\end{multline}

Indeed, for all $k\in \N$, in view of $t<0$,
\[
n(k)-t\ge n(k)\ge 1,
\]
and hence, in view of \eqref{infinity}, 
\[
e^{-[n(k)-t]\Rep\lambda_{n(k)}}\le e^{-k}.
\]

Similarly to \eqref{first2}, for any $t\ge 0$ and an arbitrary $n\in\N$,
\begin{multline}\label{second2}
\sup_{\{g^*\in X^*\,|\,\|g^*\|=1\}}
\int\limits_{\left\{\lambda\in\sigma(A)\,\middle|\,e^{t\Rep\lambda}>n\right\}}e^{t\Rep\lambda}\,dv(f,g^*,\lambda)
\\
\shoveleft{
\le \sup_{\{g^*\in X^*\,|\,\|g^*\|=1\}}e^t\sum_{k=1}^\infty e^{-[n(k)-t]\Rep\lambda_{n(k)}}
\int\limits_{\left\{\lambda\in\sigma(A)\,\middle|\,e^{t\Rep\lambda}>n\right\}\cap \Delta_{n(k)}}1\,dv(e_{n(k)},g^*,\lambda)
}\\
\shoveleft{
=e^t\sup_{\{g^*\in X^*\,|\,\|g^*\|=1\}}\sum_{k=1}^\infty e^{-\left[\frac{n(k)}{2}-t\right]\Rep\lambda_{n(k)}}
e^{-\frac{n(k)}{2}\Rep\lambda_{(k)}}
}\\
\shoveleft{
\int\limits_{\left\{\lambda\in\sigma(A)\,\middle|\,e^{t\Rep\lambda}>n\right\}\cap \Delta_{n(k)}}1\,dv(e_{n(k)},g^*,\lambda)
}\\
\hfill
\text{since, by \eqref{infinity}, there is an $L>0$ such that
$e^{-\left[\frac{n(k)}{2}-t\right]\Rep\lambda_{n(k)}}\le L$, $k\in\N$;}
\\
\shoveleft{
\le Le^t\sup_{\{g^*\in X^*\,|\,\|g^*\|=1\}}\sum_{k=1}^\infty e^{-\frac{n(k)}{2}\Rep\lambda_{n(k)}}
\int\limits_{\left\{\lambda\in\sigma(A)\,\middle|\,e^{t\Rep\lambda}>n\right\}\cap \Delta_{n(k)}}1\,dv(e_{n(k)},g^*,\lambda)
}\\
\hfill
\text{by \eqref{subvectors12};}
\\
\shoveleft{
= Le^t\sup_{\{g^*\in X^*\,|\,\|g^*\|=1\}}\sum_{k=1}^\infty
\int\limits_{\left\{\lambda\in\sigma(A)\,\middle|\,e^{t\Rep\lambda}>n\right\}\cap \Delta_{n(k)}}1\,dv(E_A(\Delta_{n(k)})h,g^*,\lambda)
}\\
\hfill
\text{by \eqref{decompose};}
\\
\shoveleft{
= Le^t\sup_{\{g^*\in X^*\,|\,\|g^*\|=1\}}
\int\limits_{\left\{\lambda\in\sigma(A)\,\middle|\,e^{t\Rep\lambda}>n\right\}}1\,dv(E_A(\cup_{k=1}^\infty\Delta_{n(k)})h,g^*,\lambda)
}\\
\hfill
\text{by \eqref{subvectors12};}
\\
\shoveleft{
=Le^t\sup_{\{g^*\in X^*\,|\,\|g^*\|=1\}}\int\limits_{\{\lambda\in\sigma(A)\,|\,e^{t\Rep\lambda}>n\}}1\,dv(h,g^*,\lambda)
\hfill
\text{by \eqref{cond(ii)};}
}\\
\shoveleft{
\le Le^t\sup_{\{g^*\in X^*\,|\,\|g^*\|=1\}}4M
\left\|E_A\left(\left\{\lambda\in\sigma(A)\,\middle|\,e^{t\Rep\lambda}>n\right\}\right)h\right\|\|g^*\|
}\\
\shoveleft{
\le 4LMe^t\|E_A(\{\lambda\in\sigma(A)\,|\,e^{t\Rep\lambda}>n\})h\|
}\\
\hfill
\text{by the strong continuity of the {\it s.m.};}
\\
\hspace{1.2cm}
\to 4LMe^t\left\|E_A\left(\emptyset\right)h\right\|=0,\ n\to\infty.
\hfill
\end{multline}

Similarly to \eqref{ffirst2}, for any $t<0$ and an arbitrary $n\in\N$,
\begin{multline}\label{ssecond2}
\sup_{\{g^*\in X^*\,|\,\|g^*\|=1\}}
\int\limits_{\left\{\lambda\in\sigma(A)\,\middle|\,e^{t\Rep\lambda}>n\right\}}e^{t\Rep\lambda}\,dv(f,g^*,\lambda)
\\
\shoveleft{
\le \sup_{\{g^*\in X^*\,|\,\|g^*\|=1\}}e^{-t}\sum_{k=1}^\infty e^{-[n(k)-t]\Rep\lambda_{n(k)}}
\int\limits_{\left\{\lambda\in\sigma(A)\,\middle|\,e^{t\Rep\lambda}>n\right\}\cap \Delta_{n(k)}}1\,dv(e_{n(k)},g^*,\lambda)
}\\
\shoveleft{
=e^{-t}\sup_{\{g^*\in X^*\,|\,\|g^*\|=1\}}\sum_{k=1}^\infty e^{-\left[\frac{n(k)}{2}-t\right]\Rep\lambda_{n(k)}}
e^{-\frac{n(k)}{2}\Rep\lambda_{(k)}}
}\\
\shoveleft{
\int\limits_{\left\{\lambda\in\sigma(A)\,\middle|\,e^{t\Rep\lambda}>n\right\}\cap \Delta_{n(k)}}1\,dv(e_{n(k)},g^*,\lambda)
}\\
\hfill
\text{since, by \eqref{infinity}, there is an $L>0$ such that
$e^{-\left[\frac{n(k)}{2}-t\right]\Rep\lambda_{n(k)}}\le L$, $k\in\N$;}
\\
\shoveleft{
\le Le^{-t}\sup_{\{g^*\in X^*\,|\,\|g^*\|=1\}}\sum_{k=1}^\infty e^{-\frac{n(k)}{2}\Rep\lambda_{n(k)}}
\int\limits_{\left\{\lambda\in\sigma(A)\,\middle|\,e^{t\Rep\lambda}>n\right\}\cap \Delta_{n(k)}}1\,dv(e_{n(k)},g^*,\lambda)
}\\
\hfill
\text{by \eqref{subvectors12};}
\\
\shoveleft{
= Le^{-t}\sup_{\{g^*\in X^*\,|\,\|g^*\|=1\}}\sum_{k=1}^\infty
\int\limits_{\left\{\lambda\in\sigma(A)\,\middle|\,e^{t\Rep\lambda}>n\right\}\cap \Delta_{n(k)}}1\,dv(E_A(\Delta_{n(k)})h,g^*,\lambda)
}\\
\hfill
\text{by \eqref{decompose};}
\\
\shoveleft{
= Le^{-t}\sup_{\{g^*\in X^*\,|\,\|g^*\|=1\}}
\int\limits_{\left\{\lambda\in\sigma(A)\,\middle|\,e^{t\Rep\lambda}>n\right\}}1\,dv(E_A(\cup_{k=1}^\infty\Delta_{n(k)})h,g^*,\lambda)
}\\
\hfill
\text{by \eqref{subvectors12};}
\\
\shoveleft{
=Le^{-t}\sup_{\{g^*\in X^*\,|\,\|g^*\|=1\}}\int\limits_{\{\lambda\in\sigma(A)\,|\,e^{t\Rep\lambda}>n\}}1\,dv(h,g^*,\lambda)
\hfill
\text{by \eqref{cond(ii)};}
}\\
\shoveleft{
\le Le^{-t}\sup_{\{g^*\in X^*\,|\,\|g^*\|=1\}}4M
\left\|E_A\left(\left\{\lambda\in\sigma(A)\,\middle|\,e^{t\Rep\lambda}>n\right\}\right)h\right\|\|g^*\|
}\\
\shoveleft{
\le 4LMe^{-t}\|E_A(\{\lambda\in\sigma(A)\,|\,e^{t\Rep\lambda}>n\})h\|
}\\
\hfill
\text{by the strong continuity of the {\it s.m.};}
\\
\hspace{1.2cm}
\to 4LMe^{-t}\left\|E_A\left(\emptyset\right)h\right\|=0,\ n\to\infty.
\hfill
\end{multline}

By Proposition \ref{prop}, \eqref{first2}, \eqref{ffirst2}, \eqref{second2}, and \eqref{ssecond2} jointly imply that 
\[
f\in \bigcap\limits_{t\in\R}D(e^{tA}),
\]
and hence, by Theorem \ref{GWS},
\[
y(t):=e^{tA}f,\ t\in\R,
\]
is a weak solution of equation \eqref{1}.

Since, for any $\lambda \in \Delta_{n(k)}$, $k\in \N$, by \eqref{radii1}, \eqref{infinity},
\begin{multline*}
\Rep\lambda =\Rep\lambda_{n(k)}-(\Rep\lambda_{n(k)}-\Rep\lambda)
\ge
\Rep\lambda_{n(k)}-|\Rep\lambda_{n(k)}-\Rep\lambda|
\\
\ \ \
\ge 
\Rep\lambda_{n(k)}-\varepsilon_{n(k)}
\ge \Rep\lambda_{n(k)}-1/n(k)\ge k-1\ge 0
\hfill
\end{multline*}
and, by \eqref{disks1},
\[
\Rep\lambda<(2n(k))^{-1}\ln|\Imp\lambda|,
\]
we infer that, for any $\lambda \in \Delta_{n(k)}$, $k\in \N$,
\begin{equation*}
|\lambda|\ge|\Imp\lambda|\ge 
e^{2n(k)\Rep\lambda}\ge e^{2n(k)(\Rep\lambda_{n(k)}-1/n(k))}.
\end{equation*}

Using this estimate, for the functional $h^*\in X^*$ defined by \eqref{functional1}, we have:
\begin{multline}\label{notin}
\int\limits_{\sigma(A)}|\lambda|\,dv(f,h^*,\lambda)
\hfill
\text{by \eqref{decompose} as in \eqref{first1};}
\\
\shoveleft{
=\sum_{k=1}^\infty e^{-n(k)\Rep\lambda_{n(k)}}\int\limits_{\Delta_{n(k)}}|\lambda|\,dv(e_{n(k)},h^*,\lambda)
}\\
\shoveleft{
\ge\sum_{k=1}^\infty e^{-n(k)\Rep\lambda_{n(k)}}e^{2n(k)(\Rep\lambda_{n(k)}-1/n(k))}v(e_{n(k)},h^*,\Delta_{n(k)})
}\\
\shoveleft{
= \sum_{k=1}^\infty
e^{-2} e^{n(k)\Rep\lambda_{n(k)}}|\langle E_A(\Delta_{n(k)})e_{n(k)},h^*\rangle|
\hfill
\text{by \eqref{infinity}, \eqref{ortho1}, and \eqref{funct-dist1};}
}\\
\hspace{1.2cm}
\ge \sum_{k=1}^\infty e^{-2}\varepsilon\dfrac{e^{n(k)}}{n(k)^2}=\infty.
\hfill
\end{multline} 

By Proposition \ref{prop}, \eqref{notin1} implies that
\[
y(0)=f\notin D(A),
\]
which, by Proposition \ref{particular} ($n=1$, $I=\{0\}$), further implies that the weak solution $y(t)=e^{tA}f$, $t\in\R$, 
of equation \eqref{1} is not strongly differentiable at $0$. 

The remaining case of
\[
\Rep\lambda_{n(k)}\to -\infty,\ k\to \infty
\] 
is symmetric to the case of
\[
\Rep\lambda_{n(k)}\to \infty,\ k\to \infty
\] 
and is considered in absolutely the same manner, which furnishes a weak solution $y(\cdot)$ of equation \eqref{1} such that
\begin{equation*}
y(0)\not\in D(A),
\end{equation*}
and hence, by Proposition \ref{particular} ($n=1$, $I=\{0\}$), not strongly differentiable at $0$.

With every possibility concerning $\{\Rep\lambda_n\}_{n=1}^\infty$ considered, 
we infer that assuming the opposite to the \textit{``if"} part's premise allows to find a weak solution of \eqref{1} on $[0,\infty)$
that is not strongly differentiable at $0$, and hence, much less strongly infinite differentiable on $\R$.

Thus, the proof by contrapositive of the \textit{``only if" part} is complete and so is the proof of the 
entire statement
\end{proof}

From Theorem \ref{real} and {\cite[Theorem $4.2$]{Markin2011}}, the latter characterizing the strong infinite differentiability of all weak solution
of equation \eqref{+} on $(0,\infty)$, we also obtain 

\begin{cor}\label{case+open}
Let $A$ be a scalar type spectral operator in a complex Banach space. If all weak solutions of equation \eqref{+} are strongly infinite differentiable on $(0,\infty)$, then all weak solutions of equation \eqref{1} are strongly infinite differentiable on $\R$.
\end{cor}

\begin{rem}
As follows from Theorem \ref{real}, all weak solutions of equation \eqref{1} with a scalar type spectral operator $A$ in a complex Banach space can 
be \textit{strongly infinite differentiable} while the operator $A$ is \textit{unbounded}, e.g., when $A$ is an unbounded \textit{self-adjoint} operator in a complex Hilbert space (cf. {\cite[Theorem $7.1$]{Markin1999}}). This fact contrasts the situation when a closed densely defined linear operator $A$ in a complex Banach space generates a strongly continuous group $\left\{T(t) \right\}_{t\in \R}$ of bounded linear operators, i.e., the associated abstract Cauchy problem is \textit{well-posed} (see Remarks \ref{remsws}), in which case even the (left or right) strong differentiability of all weak solutions of equation \eqref{1} at $0$ immediately implies \textit{boundedness} for $A$ (cf. \cite{Engel-Nagel}).
\end{rem}

\section{The Cases of Normal and Self-Adjoint Operators}

As an important particular case of Theorem \ref{real}, we obtain

\begin{cor}[The Case of a Normal Operator]\label{realnormal}\ \\
Let $A$ be a normal operator in a complex Hilbert space. Every weak solution of equation \eqref{1} is strongly infinite differentiable on $\R$ iff there exist 
$b_+>0$ and $ b_->0$ such that the set $\sigma(A)\setminus {\mathscr L}_{b_-,b_+}$,
where
\begin{equation*}
{\mathscr L}_{b_-,b_+}:=\left\{\lambda \in \C\, \middle|\,
\Rep\lambda \le \min\left(0,-b_-\ln|\Imp\lambda|\right) 
\ \text{or}\ 
\Rep\lambda \ge \max\left(0,b_+\ln|\Imp\lambda|\right)\right\},
\end{equation*}
is bounded (see Fig. \ref{fig:graph2}).
\end{cor}

\begin{rem}
Corollary \ref{realnormal} develops the results of paper \cite{Markin1999}, where similar consideration is given to the strong differentiability of the weak solutions of equation \eqref{+} with a normal operator $A$ in a complex Hilbert space on $[0,\infty)$ and $(0,\infty)$.
\end{rem}

From Corollary \ref{case+open}, we immediately obtain the following

\begin{cor}\label{casenormal+open}
Let $A$ be a normal operator in a complex Hilbert space. If all weak solutions of equation \eqref{+} are strongly infinite differentiable on $(0,\infty)$ (cf. {\cite[Theorem $5.2$]{Markin1999}}), then all weak solutions of equation \eqref{1} are strongly infinite differentiable on $\R$.
\end{cor}

Considering that, for a self-adjoint operator $A$ in a complex Hilbert space $X$,
\[
\sigma(A)\subseteq \R
\]
(see, e.g., \cite{Dun-SchII,Plesner}), we further arrive at

\begin{cor}[The Case of a Self-Adjoint Operator]\label{real self-adjoint}\ \\
Every weak solution of equation \eqref{1} 
with a self-adjoint operator $A$ in a complex Hilbert space is strongly infinite differentiable on $\R$.
\end{cor}

Cf. {\cite[Theorem $7.1$]{Markin1999}}.

\section{Inherent Smoothness Improvement Effect}

As is observed in the proof of the \textit{``only if"} part of Theorem \ref{real}, the opposite to the \textit{``if"} part's premise implies that there is a weak solution of equation \eqref{1}, which is not strongly differentiable at $0$. This renders the case of finite strong  differentiability of the weak solutions superfluous and we arrive at the following inherent effect of smoothness improvement.

\begin{prop}
Let $A$ be a scalar type spectral operator in a complex Banach space $(X,\|\cdot\|)$. If every weak solution of equation \eqref{1} is strongly differentiable at $0$, then all of them are strongly infinite differentiable on $\R$.
\end{prop}

Cf. {\cite[Proposition $5.1$]{Markin2011}}.

\section{Concluding Remark}

Due to the {\it scalar type spectrality} of the operator $A$, Theorem \ref{real} is stated exclusively in terms of the location of its {\it spectrum} in the complex plane, similarly to the celebrated \textit{Lyapunov stability theorem} \cite{Lyapunov1892} 
(cf. {\cite[Ch. I, Theorem 2.10]{Engel-Nagel}}),
and thus, is an intrinsically qualitative statement (cf. \cite{Pazy1968,Markin2011}).

\section{Acknowledgments}

The author extends sincere appreciation to his colleague, Dr.~Maria Nogin of the Department of Mathematics, California State University, Fresno, for her kind assistance with the graphics.

\section{Conflicts of Interest}

The author declares that there are no conflicts of interest regarding the publication of this paper.


\end{document}